\documentclass[11pt,a4paper]{amsart}

\usepackage{amssymb}

\newtheorem{thm}{Theorem}[section]
\newtheorem{lem}[thm]{Lemma}

\newtheorem{fact}[thm]{Fact}
\newtheorem{cor}[thm]{Corollary}

\theoremstyle{definition}
\newtheorem{defn}[thm]{Definition}

\theoremstyle{remark}

\newcommand{\GP}[2]{\langle\,#1\,\Vert\,#2\,\rangle}

\newcommand{\Th}{\operatorname{Th}}

\providecommand{\abs}[1]{\lvert#1\rvert}

\newcommand{\N}{\mathbb N}
\newcommand{\SOP}{\mathrm{SOP}}

\title{Independence Property and\\ Hyperbolic Groups}
\author{Eric Jaligot, Alexey Muranov, and Azadeh Neman}
\date{\today}
\address{Universit\'e de Lyon ;
Universit\'e Lyon 1 ;
Institut Camille Jordan CNRS UMR 5208 ;
43, boulevard du 11 novembre 1918,
F--69622 Villeurbanne Cedex, France}


\begin{document}

\maketitle

\begin{abstract}
In continuation of \cite{HoucineJaligot04,Houcine2007},
we prove that existentially closed $CSA$-groups have 
the independence property.
This is done by showing that there exist words having
the independence property relative to the class of 
torsion-free hyperbolic groups. 
\end{abstract}

\section{Introduction}

A \emph{$CSA$-group} is a group $G$ in which every maximal
abelian subgroup $A$ is malnormal, i.e.\ satisfies $A\cap A^{g}=\{1\}$
for each element $g$ of $G$ which is not in $A$.
The interest in these groups is motivated on the one hand by the question 
of existence of new groups with well behaved first-order theories, 
and ultimately of so-called \emph{bad groups} 
of finite Morley rank which may be $CSA$-groups
under certain circumstances \cite{Jaligot01}.
On the other hand, it is motivated by the fact that $CSA$-groups
are usually more suitable than other groups for solving equations, 
an important step towards understanding their first-order theory.

The torsion of $CSA$-groups under consideration
can be restricted as follows. 
If $f$ is a function from the set of prime integers into 
$\N\sqcup\{\infty\}$, we call a \emph{$CSA_{f}$-group} any $CSA$-group 
which contains no elementary abelian $p$-subgroup of 
rank $f(p)+1$ for any prime $p$ such that $f(p)$ is finite. 
Once such a function $f$ is fixed, it is easily seen that the
first-order class of $CSA_{f}$-groups is inductive,
and thus contains \emph{existentially closed} $CSA_{f}$-groups.
If a $CSA$-group contains an involution, then it must be abelian, 
a case of low interest from the group-theoretic point of view and from 
the model-theoretic point of view as well, since all abelian groups
are known to be stable. Hence we work with $CSA$-groups without 
involutions, i.e.\ with $CSA_{f}$-groups where $f(2)=0$.
It is known, when $f(2)=0$, that existentially closed $CSA_{f}$-groups
are simple, in the strong sense that simplicity is provided
by a first-order formula.
Moreover, their maximal abelian subgroups are 
conjugate, divisible, and of Pr\"ufer $p$-rank $f(p)$ for each prime $p$. 
We refer to \cite[Sect.\ 2, 5]{HoucineJaligot04} for details. 

Concerning the model theory of existentially closed $CSA_{f}$-groups, 
always assuming $f(2)=0$, it has also been proved
in \cite{HoucineJaligot04} that such groups are not $\omega$-stable. 
This was done by counting the number of types, more precisely by showing 
the existence of $2^{\aleph_{0}}$ types over the empty set.
The same method of counting types has later been used in 
\cite{Houcine2007} to prove that the first-order theory of
existentially closed $CSA_{f}$-groups ($f(2)=0$) is not superstable.
We will show here by a much more elementary argument 
that their first-order theory is indeed far from being stable.
We shall prove that 
they have the independence property using a standard small cancellation 
argument from combinatorial group theory.

\section{Relative independence property} 

The main result of the present paper leading to the independence property of 
existentially closed $CSA$-groups is the following.

\begin{thm}
\label{MainTheo}
There exists a group word $w(x,y)$ in two variables such that the formula
$\ulcorner w(x,y)=1\urcorner$ has the independence property
relative to the class of torsion-free hyperbolic groups. 
\end{thm}

For the moment we postpone the proof of Theorem \ref{MainTheo}, and rather give 
the definition used in this theorem and derive the corollary concerning the 
first-order theories of existentially closed $CSA$-groups which particularly interests us. 

Hyperbolic groups have been introduced by Gromov in \cite{Gromov87} and have 
been intensively studied in geometric group theory 
since \cite{{CoornaertDelzantPapadopoulos1990}, {GhysdelaHarpe1990}}. 
They are finitely presented groups for which, by definition, the Cayley 
graph ``looks like" a tree. In this sense they are finitely generated groups with 
only ``few" relations between the generators, and hence can be seen as a generalization 
of finitely generated free groups. 
For example, by Fact \ref{FactTFHypImpliesCSA} below, 
torsion-free hyperbolic groups are $CSA$-groups, an elementary property of free groups. 

If $\phi(\overline{x},\overline{y})$ is a formula in a given language
$\mathcal L$, and $\mathcal C$ is any (not necessarily elementary)
class of $\mathcal L$-structures, we say that $\phi$ has the 
\emph{independence property relative to $\mathcal C$} if
for any positive integer $n$,
there exists a structure $M_{n}$ in $\mathcal C$ with sequences 
of tuples 
$\overline{x}_{1}, \dots , \overline{x}_{i}, \dots , \overline{x}_{n}$, 
and 
$\overline{y}_{\sigma}$, 
where there are $2^{n}$ indices $\sigma$ varying over the set of subsets of 
$\{ 1,\dots,n\}$, such that the following holds in $M_{n}$:
$$
M_{n}\models\phi(\overline{x}_{i},\overline{y}_{\sigma})
\quad\text{if and only if}\quad
i\in\sigma.
$$
When the class $\mathcal C$ consists of all models of a complete 
first-order theory $T$, this definition corresponds to the usual notion of 
independence property of the formula $\phi$ relative to the
first-order theory $T$, as defined in \cite{Shelah90}.
As we are dealing with groups here, the language 
will be that of groups, and our formula will typically be a
group equation in two variables, i.e.\ of the form
$\ulcorner w(x,y)=1\urcorner$ with
$w(x,y)$ a \emph{group word} in two variables $x$ and $y$
(i.e.\ a word in $x^{\pm1}$ and $y^{\pm1}$). 

Before the proof of Theorem \ref{MainTheo}, we derive the corollary 
concerning the model theory of existentially closed $CSA$-groups. 

\begin{cor}
\label{CorECCSAGpsHaveIP}
Assume $f(2)=0$ and let $G$ be an existentially closed $CSA_{f}$-group, or more 
generally a group having the same universal theory as an existentially 
closed $CSA_{f}$-group. 
Then the first-order theory of $G$ has the independence property. 
\end{cor}

For deriving Corollary \ref{CorECCSAGpsHaveIP} from Theorem \ref{MainTheo},
we will use the following fact.

\begin{fact}
\label{FactCSAEmbedsECCSA}
Assume $f(2)=0$ and let $G$ be a group having the same universal theory as 
an existentially closed $CSA_{f}$-group. 
Then any $CSA_{f}$-group embeds into a model of the first-order 
theory of $G$. 
\end{fact}

\begin{proof}
It was explicitly seen in the proof of 
\cite[Corollary 8.2]{HoucineJaligot04} that the universal theory of an 
existentially closed $CSA_{f}$-group is true in any $CSA_{f}$-group. 
Now the fact that the models of the universal part $T_{\forall}$ of a first-order theory $T$ 
are precisely the substructures of models of $T$ \cite[Corollary 6.5.3]{Hodges(book)93} 
yields our claim. 
\end{proof}

We will also use the following fact, according to which $CSA$-groups 
can be seen as a generalization of torsion-free hyperbolic groups.

\begin{fact}[{\cite[Proposition 12]{MyasnikovRemeslenikov96}}]
\label{FactTFHypImpliesCSA}
Every torsion-free hyperbolic group is $CSA$.
\end{fact}

We shall give here a short proof for reader's convenience. 
By \cite[Corollaire 7.2]{CoornaertDelzantPapadopoulos1990} or 
\cite[Th\'eor\`eme 38]{GhysdelaHarpe1990}, the centralizer of any element 
of infinite order in a hyperbolic group is virtually cyclic, i.e. has a cyclic subgroup of 
finite index. Hence Fact \ref{FactTFHypImpliesCSA} is indeed a special case of 
the following purely group theoretic lemma. 

\begin{lem}\label{DoubleLemma}
Let $G$ be a torsion-free group. 
\begin{itemize}
\item[\textrm{a.}]
If $G$ is virtually cyclic, then $G$ is cyclic. 
\item[\textrm{b.}]
If centralizers of nontrivial elements of $G$ are cyclic, then $G$ is $CSA$. 
\end{itemize}
\end{lem}

\begin{proof}\
\textrm{a.}\ 
Assume $G$ is virtually cyclic. 
First, it is easy to verify that in every torsion-free virtually cyclic group 
the center has finite index. 
(This is true even for virtually cyclic groups without involutions.) 
Second, there is a theorem, often attributed to Issai Schur,
which states that if the index of the center of a group is finite,
then the derived subgroup of the group is also finite.
In particular, if in a torsion-free group the center has finite index,
then the group is abelian.
An elementary proof of this theorem of Schur
may be found in \cite{Rosenlicht1962}:
the idea is to show that in a group whose center has index $n$,
the number of commutators is at most $(n-1)^2+1$, and that the $(n+1)$st
power of each commutator is the product of $n$ commutators.
Hence $G$ is abelian. Now it follows that $G$ is cyclic from 
the classification of finitely generated abelian groups,
or from the simpler fact that in a torsion-free abelian group, each two
cyclic subgroups either intersect trivially,
or are both contained in some cyclic subgroup.

\textrm{b.}\
Assume now centralizers of nontrivial elements of $G$ are cyclic. 
In particular, commuting is an equivalence relation on $G\setminus\{1\}$.
Consider an arbitrary maximal abelian subgroup $A$ of $G$.
Then $A$ is infinite cyclic.
Suppose $A$ is not malnormal.
Let $a$ be a generator of $A$, and let $b$ be an element of
$G\setminus A$ such that $A\cap A^{b}\ne\{1\}$. 
Since commutativity is transitive and $A$ is maximal, $A=A^{b}$.
Therefore, $a^b\in\{a^{\pm1}\}$, and hence $(a^b)^b=a$.
Since $a$ commutes with $b^2$, and $b^2$ commutes with $b$,
we have that $a$ commutes with $b$, and therefore $b\in A$,
which gives a contradiction.
Thus $A$ is malnormal.
Hence $G$ is $CSA$. 
\end{proof}

We remark that, in addition to Fact \ref{FactTFHypImpliesCSA},
Lemma \ref{DoubleLemma} implies that all maximal
abelian subgroups of a torsion-free hyperbolic group are cyclic.

\begin{proof}[Proof of Corollary\/ \textup{\ref{CorECCSAGpsHaveIP}.}]
By Theorem \ref{MainTheo}, there exists a group word $w(x,y)$ in two 
variables $x$ and $y$ such that the formula $\ulcorner w(x,y)=1\urcorner$
has the independence property relative to the class of
torsion-free hyperbolic groups. 
This means that for any positive integer $n$,
there exists a torsion-free hyperbolic group $G_{n}$ with sequences 
of elements $x_{1},\dots,x_{i},\dots,x_{n}$, and 
$y_{\sigma}$, with 
$\sigma\subseteq{\{1,\dots,n\}}$, such that 
$$
G_{n}\models w(x_{i},y_{\sigma})=1
\quad\text{if and only if}\quad
i\in \sigma.
$$
By Fact \ref{FactTFHypImpliesCSA}, any torsion-free hyperbolic 
group $G_{n}$ is a $CSA$-group, and even, as it is torsion-free, 
a $CSA_{f}$-group for an arbitrary function $f$. 

Let now $\Th(G)$ denote the first-order theory of a group $G$ having the same 
universal theory as an existentially closed $CSA_{f}$-group. 
By Fact \ref{FactCSAEmbedsECCSA}, each group $G_{n}$ embeds into 
a model of $\Th(G)$.
In particular, as the truth of the formula $\ulcorner w(x,y)=1\urcorner$
is preserved 
under embeddings, $\Th(G)$ contains the formula 
$$
(\exists_{1\leq i \leq n}\;x_{i})\,
(\exists_{\sigma\subseteq\{1,\dots,n\}}\;y_{\sigma})\,
\bigl[(\bigwedge_{i \in \sigma}\;w(x_{i},y_{\sigma})=1)
\wedge
(\bigwedge_{i \notin \sigma}\;w(x_{i},y_{\sigma})\neq 1)\bigr].
$$
Since this is true for any positive $n$,
the formula $\ulcorner w(x,y)=1\urcorner$ has the independence property
relative to 
$\Th(G)$.
Hence $\Th(G)$ has the independence property. 
\end{proof}

\section{Small-cancellation groups}

In this section we consider groups together with their presentations
by generators and defining relations.
We take most definitions and results from
\cite{OlshanskiiBook1991}, which is our main source of references
for the combinatorial treatment of group presentations.
Some terms are also borrowed from \cite{CCH81}.

Recall that a \emph{group presentation} is an ordered pair
$\GP{A}{\mathcal R}$,
also written as $\GP{A}{R=1;R\in\mathcal R}$, where
$A$ is an arbitrary set, called the \emph{alphabet}, and $\mathcal R$
is a set of \emph{group words} over $A$, called \emph{relators} or
\emph{defining relators}.
Every relator from $\mathcal R$ is a word in the alphabet $A\sqcup A^{-1}$,
composed of elements of $A$ and their formal inverses.
Relators can also be viewed as a simplified form of
\emph{terms} in the language of groups augmented with constants from~$A$.

By abuse of notation, the same symbol $\GP{A}{\mathcal R}$ shall be used
to denote a certain group given by this presentation, that is
a group $G$ where all constants from $A$ are interpreted
in such a way that $G$ is generated by their interpretations,
all the relations $\ulcorner R=1\urcorner$, $R\in\mathcal R$, hold in $G$,
and all other relations of this form that hold in $G$ are just
their group-theoretic consequences.
A presentation of a group $G$ is any presentation $\GP{A}{\mathcal R}$
such that $G\cong\GP{A}{\mathcal R}$.
Every group has a presentation (by the multiplication table, for example).

A presentation $\GP{A}{\mathcal R}$ is called \emph{finite}
if both sets $A$ and $\mathcal R$ are finite.
A group is \emph{finitely presented} if it has a finite presentation.

Since \emph{cyclic reduction} of relators, i.e.\ \emph{cyclic shifts}
and cancellation of pairs of adjacent mutually inverse letters, 
does not change the group given by the presentation,
every group has a presentation with all relators \emph{cyclically reduced}.
Presentations, or sets of relators,
in which all relators are cyclically reduced shall be called
\emph{cyclically reduced} themselves.

A cyclically reduced presentation $\GP{A}{\mathcal R}$,
or a set of relators $\mathcal R$,
shall be called \emph{symmetrized} if
$\mathcal R$ contains the \emph{visual inverse} $R^{-1}$ and every
\emph{cyclic shift} $YX$ of every $R\equiv XY\in\mathcal R$.
(The visual inverse of $ab^{-1}c$, for example, is $c^{-1}ba^{-1}$;
$XY$ denotes the concatenation of the words $X$ and $Y$.)
The \emph{symmetrized closure}, or \emph{symmetrization},
of $\mathcal R$ is the minimal symmetrized set of relators
containing $\mathcal R$.

The opposite of ``symmetrized'' is ``concise'': 
following \cite{CCH81}, we call a cyclically reduced presentation
$\GP{A}{\mathcal R}$, or a set of relators $\mathcal R$,
\emph{concise} if no two distinct elements of $\mathcal R$
are cyclic shifts of each other or of each other's visual inverses.

In what follows, all group presentations shall be assumed
cyclically reduced. The length of a word $X$ shall be denoted 
by $\abs{X}$. 

If $R_{1}\equiv XY_1$ and $R_{2}\equiv XY_2$ are two distinct words in
a symmetrized set of relators $\mathcal R$, then $X$ is called
a \emph{piece} (of $R_1$) relative to~$\mathcal R$.

\begin{defn}
Let $\lambda$ be a number in $[0,1]$.
Let $\GP{A}{\mathcal R}$ be a group presentation,
and $\tilde{\mathcal R}$ be the symmetrization of $\mathcal R$.
Then the set $\mathcal R$, or the presentation $\GP{A}{\mathcal R}$,
is said to satisfy the \emph{condition} $C'(\lambda)$ if
$\abs{X}<\lambda\abs{R}$ 
for any $R\in\tilde{\mathcal R}$ and for any subword $X$ of $R$ which is
a piece relative to $\tilde{\mathcal R}$.
\end{defn}

The condition $C'(\lambda)$ with ``small'' $\lambda$
is a standard example of a so-called ``small cancellation condition''.
There are various such small cancellation conditions and generally
such conditions allow one to prove that
large traces of relators remain in all their consequences.

The following fact about the condition $C'(\lambda)$
shows that it can be used for constructing hyperbolic groups.

\begin{fact}[{\cite[Th\'eor\`eme 33]{GhysdelaHarpe1990}}]
\label{FactC'1/6ImpliesHyp}
A finitely presented group which has a presentation satisfying $C'(1/6)$
is hyperbolic.
\end{fact}

As an intermediate step of our proof of Theorem \ref{MainTheo}, 
we are going to use \emph{asphericity\/} of a group presentation.
Various forms of asphericity and their mutual implications are studied
in \cite{CCH81}.
What we will actually need is the asphericity defined in
\cite[\S13]{OlshanskiiBook1991}, 
though it is defined there in a more general setting of
\emph{graded\/} presentations.
To avoid confusion with other forms of asphericity, however, we will borrow 
the term \emph{diagrammatic asphericity\/} from \cite{CCH81}.
(It follows from \cite[Theorem 32.2]{OlshanskiiBook1991}
that the asphericity defined in \cite[\S13]{OlshanskiiBook1991} in the case
of a non-graded cyclically reduced presentation is indeed the same as
diagrammatic asphericity in \cite{CCH81}.)
Diagrammatic asphericity of a group presentation means that there exists no 
reduced spherical van Kampen diagram over that presentation.
As we do not use the definition of asphericity in our argument,
we shall not explain it any further.

Taking small values of $\lambda$ in the condition
$C'(\lambda)$ results in diagrammatic asphericity.
Indeed, the $C'(1/5)$ condition implies a related small cancellation
condition $C(6)$ \cite[p. 127]{OlshanskiiBook1991}, 
which yields diagrammatic asphericity by
\cite[Theorem 13.3]{OlshanskiiBook1991}.
Hence we have the following fact.

\begin{fact}
\label{FactC'1/5ImpliesAspherical}
A presentation satisfying $C'(1/5)$ is diagrammatically aspherical.
\end{fact}

Diagrammatic asphericity of a group presentation implies many algebraic properties 
of the group. An important consequence is the following.

\begin{fact}[{\cite[Theorem 13.4]{OlshanskiiBook1991}}]
\label{FactAsphericIndep}
If $\GP{A}{\mathcal R}$ is a concise diagrammatically aspherical
group presentation,
then the defining relators are independent, i.e.\ no relation in the set 
$\{\,\ulcorner R=1\urcorner\,|\,R\in\mathcal R\,\}$
is a consequence of the others. 
\end{fact} 

Following \cite{CCH81} and \cite[Sect.\ 6]{Muranov2007},
we say that a group presentation $\GP{A}{\mathcal R}$
is \emph{singularly aspherical} if it is diagrammatically aspherical,
concise, and no element of $\mathcal R$ can be decomposed
as a concatenation of several
copies of the same subword
(i.e.\ does not represent a proper power in the free group
$\GP{A}{\varnothing}$). 

If $\GP{A}{\mathcal R}$ is a singularly aspherical presentation,
then the relation module of $\GP{A}{\mathcal R}$ is a free $G$-module by 
\cite[Corollary 32.1]{OlshanskiiBook1991}.
One can see then, using the fact that all
odd-dimensional homology groups of nontrivial finite cyclic groups
are nontrivial \cite{Brown94}, that this prevents the presence
of torsion. 

\begin{fact}[{\cite[Lemma 64]{Muranov2007}}]
\label{FactGSingAsphericTorsionFree}
Every group with singularly aspherical presentation is torsion-free. 
\end{fact}

Note that Fact \ref{FactGSingAsphericTorsionFree} can also be 
deduced from \cite[Theorem 3]{HUEB79},
where the finite-order elements are classified in the more general case of 
\emph{combinatorially aspherical} groups (defined in \cite{HUEB79,CCH81}). 

We have now all the tools from combinatorial group theory needed to 
build families of finitely presented groups with the desired properties. 

\begin{proof}[Proof of Theorem\/ \textup{\ref{MainTheo}.}]
Consider, for example, the following word $w(x,y)$ in two letters
$x$ and $y$:
$$
w(x,y)=xy^{7}x^{2}y^{6}x^{3}y^{5}x^{4}y^{4}x^{5}y^{3}x^{6}y^{2}x^{7}y.
$$
For an integer $n\geq 1$, let $A$ be a set of $n+2^{n}$ elements 
$a_{1}, \dots , a_{i}, \dots , a_{n}$, and 
$b_{\sigma}$, with $2^{n}$ indices 
$\sigma$ varying in the set of all subsets of $\{1,\dots,n\}$.
Define two sets of relators:
$$
\mathcal R=\{\,w(a_{i},b_{\sigma})\,|\,1\leq i \leq n\,;\,
\sigma\subseteq\{1,\dots,n\}\,;\,
i\in\sigma\,\}$$
and 
$$
\mathcal S=\{\,w(a_{i},b_{\sigma})\,|\,1\leq i \leq n\,;\,
\sigma\subseteq\{1,\dots,n\}\,\}.
$$
Denote by $\tilde{\mathcal R}$ and $\tilde{\mathcal S}$
their respective symmetrizations.
Clearly both $\mathcal R$ and $\mathcal S$ are concise.

Consider now the finitely presented group 
$$
G=\GP{A}{\mathcal R}=\GP{A}{\tilde{\mathcal R}},
$$ 
and denote by $F$ the free group $\GP{A}{\varnothing}$, of rank $n+2^{n}$.
Then $G$ is the quotient of $F$ by the normal closure 
of $\mathcal R$, $\mathcal R$ being viewed as a subset of~$F$.

One can check directly that any piece relative to $\tilde{\mathcal S}$
of any relator from $\tilde{\mathcal S}$ is either of the form 
$(a_i^mb_\sigma^n)^{\pm1}$, or of the form $(b_\sigma^na_i^m)^{\pm1}$. 
Then we deduce that the length of such a piece is at most $8$.
As the length of every relator is $7\times8$, we obtain 
$$
\abs{X}\leq\frac{1}{7}\abs{R}<\frac{1}{6}\abs{R}
$$
for any $R\in\tilde{\mathcal S}$ and any piece $X$ of $R$
relative to $\tilde{\mathcal S}$.
Hence the presentation $\GP{A}{\mathcal S}$ satisfies $C'(1/6)$.
In particular, the subpresentation $\GP{A}{\mathcal R}$ satisfies $C'(1/6)$.

It follows now from Fact \ref{FactC'1/6ImpliesHyp} that
$G$ is hyperbolic.
By Fact \ref{FactC'1/5ImpliesAspherical}, the presentation
$\GP{A}{\mathcal S}$ is diagrammatically aspherical.
As clearly no element of $\mathcal S$ is a proper power in 
the free group $F$,
the presentation $\GP{A}{\mathcal S}$ is in fact
singularly aspherical, and such is the subpresentation
$\GP{A}{\mathcal R}$.
Fact \ref{FactGSingAsphericTorsionFree} implies now
that $G$ is torsion-free. 

Finally, applying Fact \ref{FactAsphericIndep} to $\GP{A}{\mathcal S}$,
we obtain that for all $i\in\{1,\dots,n\}$ and
$\sigma\subseteq\{1,\dots,n\}$,
$$
G\models w(a_{i},b_{\sigma})=1
\quad\text{if and only if}\quad i\in\sigma. 
$$

As $n$ can be any positive integer,
the formula $\ulcorner w(x,y)=1\urcorner$ has the independence property
relative to the class of torsion-free hyperbolic groups. 
The proof of Theorem \ref{MainTheo} is now complete. 
\end{proof}

\section{Concluding remarks}

We conclude with a few remarks.

\subsection{}
The group word $w(x,y)$ built in the proof of Theorem \ref{MainTheo} is in 
two variables and without additional constants.
Hence in Corollary \ref{CorECCSAGpsHaveIP} one finds
directly a formula with the independence property having two single
variables only 
(see \cite[Th\'eor\`eme 12.18]{Poizat(book)85} for general reductions
of this kind), 
and without parameters. 

\subsection{}
As announced in \cite{Sela2007}, a torsion-free hyperbolic group is stable, 
and hence cannot have the independence property.
Hence one cannot imagine a version of 
Theorem \ref{MainTheo} where the class of groups would consist of 
groups elementarily equivalent to a given torsion-free hyperbolic group, 
or more generally to a fixed \emph{finite}
set of torsion-free hyperbolic groups. 
Our proof provides however a countable set of torsion-free
hyperbolic groups. 

\subsection{}
The proof of Theorem \ref{MainTheo} actually provides a group word
$w(x,y)$ such 
that the formula $\ulcorner w(x,y)=1\urcorner$
has the independence property relative to 
the class of torsion-free finitely presented $C'(1/6)$-groups,
a class significantly smaller than that of all
torsion-free hyperbolic groups. 
Choosing $w$ long enough, one can similarly produce formulas having the 
independence property relative to the class of
torsion-free finitely presented $C'(\lambda)$-groups
with $\lambda$ arbitrarily small. 
Indeed, if one denotes by 
$P_{n}$ the ``probability'' that a cyclically reduced group word $w(x,y)$ 
of length $n\ge1$ gives the independence property relative to 
the class of torsion-free $C'(\lambda)$-groups 
(with the formula $\ulcorner w(x,y)=1\urcorner$),
then one can check in a rather straightforward manner
the rough estimate
$$
P_{n}\ge1-\frac{n^2}{2^{\lambda n}}-\frac{4n}{2^{\lambda n}},
$$
and hence $P_{n}$ tends rapidly to $1$ as $n$ tends to the infinity. 
This estimate can be found by verifying that the ``probability'' of 
occurrence of a subword of a cyclic shift of $w$ of length $\lceil\lambda n\rceil$ 
in two given distinct ``positions'' with respect to $w$ 
is not greater than $1/2^{\lambda n}$, and that the
``probability'' that a cyclic 
shift of $w$ contains a syllable of length $\lceil\lambda n\rceil$
is less than ${4n}/{2^{\lambda n}}$.
For example, the probability that in a cyclically reduced
group word $w(x,y)$ of length at least $6$
the same word of length $4$ occures as a subword
starting from the first letter and also starting from third letter,
is at most $1/(3\cdot3\cdot3\cdot2)=1/54$; indeed,
once all the letters of $w$ except the first $4$ are fixed, there
is at most one way to complete $w$ so as to obtain
a word whose initial subword of length $4$ occures again starting
from the third letter,
but there are at least $3\cdot3\cdot3\cdot2$
ways to complete it to obtain a cyclically reduced group word.

Roughly speaking, for any fixed $\lambda>0$,
by choosing an arbitrary
but sufficiently long group word $w$, one will ``most probably''
obtain the formula $\ulcorner w(x,y)=1\urcorner$ with
the independence property relative to the class of torsion-free
finitely presented $C'(\lambda)$-groups. 

\subsection{}
Another property of a formula witnessing the instability of a 
first-order theory, orthogonal in some sense to 
the independence property, is the \emph{strict order property}.
One may wonder whether the first-order theory
$\Th(G)$ of an existentially closed $CSA$-group $G$ has this property. 
We provide here some speculations on this question, 
again for formulas $\phi(x,y)$ of the 
form $\ulcorner w(x,y)=1\urcorner$ for some group word $w$,
and we concentrate on weaker 
properties such as the \emph{$n$-strong order property $\SOP_{n}$}
defined in \cite[Definition 2.5]{Shelah1996} for $n\geq3$. 
We refer to \cite[Sect.\ 2]{Shelah1996} for a general discussion
of these properties, and we just recall the implications
$$
\begin{array}{rcll}
\text{strict order property} & \implies & \cdots \\
& \implies & \SOP_{n+1}\\
& \implies & \SOP_{n}\\ 
& \implies & \cdots\\
& \implies & \SOP_{3} & \implies \text{non-simplicity} \\
\end{array}
$$
If $w$ is a ``long'' and ``arbitrarily chosen'' cyclically reduced
group word in two letters $x$ and $y$ (see above),
then ``most probably'' the formula
$\ulcorner w(x,y)=1\urcorner$ will not exemplify $\SOP_{n}$. 
The reason is that essentially 
the same kind of arguments used in the proof of 
Theorem \ref{MainTheo} provide torsion-free $C'(1/6)$-groups 
generated by elements $a_{1}$, \dots, $a_{n}$ in which 
$w(a_{i},a_{j})=1$ if and only if $j=i+1$ modulo $n$,
and thus the graph defined on an existentially closed $CSA_{f}$-group
by the formula $\ulcorner w(x,y)=1\urcorner$ contains cycles of size $n$,
contrary to one of the two requirements 
for the condition $\SOP_{n}$
(the other one being the existence in some model
of an infinite chain in this graph).
Hence one may be tempted to look at ``short'' words. 
As usual in our context of commutative transitive groups,
the very short word $[x,y]$ witnessing the commutation 
of $x$ and $y$ boils down to an equivalence relation, and hence is useless. 
Incidentally, the formula $\phi(x,y)$ used 
in \cite[Proposition 4.1]{ShelahUsvyatsov} to 
prove that Group Theory ``in general" has $\SOP_{3}$ is 
$$
(xyx^{-1}=y^{2}) \wedge (x\neq y).
$$
In our context of $CSA$-groups, it implies $y=1$. 
Hence the absence of certain triangles $(a_{1},a_{2},a_{3})$ satisfying 
$$
\phi(a_{1},a_{2}) \wedge \phi(a_{2},a_{3}) \wedge \phi(a_{3},a_{1})
$$
(provided by \cite[p. 493]{Stallings1991} in the context of
arbitrary groups) is immediate in our case,
but one cannot hope to find an infinite chain in the graph 
associated to $\phi(x,y)$.
Hence this formula could not exemplify the $\SOP_{3}$
in our context of $CSA$-groups. 
This means that a formula examplifying the $\SOP_{n}$ of
an existentially closed $CSA$-group, if it exists,
cannot involve only an ``arbitrarily chosen long'' equation,
and it does not seem to involve only ``short" equations. 

\section*{Acknowledgments}

We thank Abderezak Ould Houcine who pointed out fallacies
in preliminary attempts towards the present work. 


\begin{thebibliography}{GdlH90}

\bibitem[Bro94]{Brown94}
K.~S. Brown.
\newblock {\em Cohomology of groups}, volume~87 of {\em Graduate Texts in
  Mathematics}.
\newblock Springer-Verlag, New York, 1994.
\newblock Corrected reprint of the 1982 original.

\bibitem[CCH81]{CCH81}
I.~M. Chiswell, D.~J. Collins, and J.~Huebschmann.
\newblock Aspherical group presentations.
\newblock {\em Math. Z.}, 178(1):1--36, 1981.

\bibitem[CDP90]{CoornaertDelzantPapadopoulos1990}
M.~Coornaert, T.~Delzant, and A.~Papadopoulos.
\newblock {\em G\'eom\'etrie et th\'eorie des groupes}, volume 1441 of {\em
  Lecture Notes in Mathematics}.
\newblock Springer-Verlag, Berlin, 1990.
\newblock Les groupes hyperboliques de Gromov. [Gromov hyperbolic groups], With
  an English summary.

\bibitem[GdlH90]{GhysdelaHarpe1990}
{\'E}.~Ghys and P.~de~la Harpe, editors.
\newblock {\em Sur les groupes hyperboliques d'apr\`es {M}ikhael {G}romov},
  volume~83 of {\em Progress in Mathematics}.
\newblock Birkh\"auser Boston Inc., Boston, MA, 1990.
\newblock Papers from the Swiss Seminar on Hyperbolic Groups held in Bern,
  1988.

\bibitem[Gro87]{Gromov87}
M.~Gromov.
\newblock Hyperbolic groups.
\newblock In {\em Essays in group theory}, volume~8 of {\em Math. Sci. Res.
  Inst. Publ.}, pages 75--263. Springer, New York, 1987.

\bibitem[Hod93]{Hodges(book)93}
W.~Hodges.
\newblock {\em Model theory}, volume~42 of {\em Encyclopedia of Mathematics and
  its Applications}.
\newblock Cambridge University Press, Cambridge, 1993.

\bibitem[Hue79]{HUEB79}
J.~Huebschmann.
\newblock Cohomology theory of aspherical groups and of small cancellation
  groups.
\newblock {\em J. Pure Appl. Algebra}, 14(2):137--143, 1979.

\bibitem[Jal01]{Jaligot01}
E.~Jaligot.
\newblock Full {F}robenius groups of finite {M}orley rank and the
  {F}eit-{T}hompson theorem.
\newblock {\em Bull. Symbolic Logic}, 7(3):315--328, 2001.

\bibitem[JOH04]{HoucineJaligot04}
E.~Jaligot and A.~Ould~Houcine.
\newblock Existentially closed {CSA}-groups.
\newblock {\em J. Algebra}, 280(2):772--796, 2004.

\bibitem[Mur07]{Muranov2007}
A.~Yu. Muranov.
\newblock Finitely generated infinite simple groups of infinite commutator
  width.
\newblock {\em Internat.\ J.\ Algebra Comput.}, 17(3):607--659, 2007.
\newblock arXiv.org preprint:
  \href{http://arxiv.org/abs/math.GR/0608688}{math.GR/0608688}.

\bibitem[MR96]{MyasnikovRemeslenikov96}
A.~G. Myasnikov and V.~N. Remeslennikov.
\newblock Exponential groups. {II}. {E}xtensions of centralizers and tensor
  completion of {CSA}-groups.
\newblock {\em Internat. J. Algebra Comput.}, 6(6):687--711, 1996.

\bibitem[Ol{\cprime}91]{OlshanskiiBook1991}
A.~Yu. Ol{\cprime}shanski{\u\i}.
\newblock {\em Geometry of defining relations in groups}, volume~70 of {\em
  Mathematics and its Applications (Soviet Series)}.
\newblock Kluwer Academic Publishers Group, Dordrecht, 1991.
\newblock Translated from the 1989 Russian original by Yu.\ A. Bakhturin.

\bibitem[OH07]{Houcine2007}
A.~Ould~Houcine.
\newblock On superstable ${CSA}$-groups.
\newblock To appear in Annals of Pure and Applied Logic,
  http://math.univ-lyon1.fr/$\sim$ould/, 2007.

\bibitem[Poi85]{Poizat(book)85}
B.~Poizat.
\newblock {\em Cours de th\'eorie des mod\`eles}.
\newblock Bruno Poizat, Villeurbanne, 1985.
\newblock Une introduction \`a la logique math\'ematique contemporaine. [An
  introduction to contemporary mathematical logic].

\bibitem[Ros62]{Rosenlicht1962}
M.~Rosenlicht.
\newblock On a result of {B}aer.
\newblock {\em Proc. Amer. Math. Soc.}, 13(1):99--101, 1962.

\bibitem[Sel07]{Sela2007}
Z.~Sela.
\newblock Diophantine geometry over groups {VIII}: Stability.
\newblock preprint: http://www.ma.huji.ac.il/$\sim$zlil/, 2007.

\bibitem[She90]{Shelah90}
S.~Shelah.
\newblock {\em Classification theory and the number of nonisomorphic models}.
\newblock North-Holland Publishing Co., Amsterdam, second edition, 1990.

\bibitem[She96]{Shelah1996}
S.~Shelah.
\newblock Toward classifying unstable theories.
\newblock {\em Ann. Pure Appl. Logic}, 80(3):229--255, 1996.

\bibitem[SU06]{ShelahUsvyatsov}
S.~Shelah and A.~Usvyatsov.
\newblock Banach spaces and groups---order properties and universal models.
\newblock {\em Israel J. Math.}, 152:245--270, 2006.

\bibitem[Sta91]{Stallings1991}
J.~R. Stallings.
\newblock Non-positively curved triangles of groups.
\newblock In {\em Group theory from a geometrical viewpoint (Trieste, 1990)},
  pages 491--503. World Sci. Publ., River Edge, NJ, 1991.

\end{thebibliography}

\def\cprime{$'$} \providecommand{\href}[2]{#2}

\end{document}